\documentclass[oneside,a4paper,10pt]{amsart}
\newcommand{\thetitle}{On the geometry of the set of symmetric matrices with repeated eigenvalues}

\newcommand{\theauthor}{Paul Breiding, Khazhgali Kozhasov and Antonio Lerario}
\newcommand{\citecolor}{black}
\newcommand{\linkcolor}{black}

\newcommand{\txt}[1]{\text{\normalfont{#1}}}
\newcommand{\R}{\mathbb{R}}
\newcommand{\C}{\mathbb{C}}

\usepackage[utf8]{inputenc}
\usepackage{listings}
\usepackage[T1]{fontenc}
\usepackage[english]{babel}
\usepackage[babel]{csquotes}
\usepackage{tikz}
\usepackage{amsthm}
\usepackage{amsmath,amssymb}
\usepackage[title,titletoc]{appendix}

\usepackage{mathtools} %scommentare per far numerare solo le equationi che sono chiamate nel testo
\mathtoolsset{showonlyrefs,showmanualtags} %scommentare per far numerare solo le equationi che sono chiamate nel testo
\usepackage{mathtools}
\mathtoolsset{showonlyrefs=false}
\usepackage[format=plain, font=footnotesize]{caption}
\usepackage[colorlinks=true, citecolor=\citecolor, linkcolor=\linkcolor]{hyperref}
\usepackage{cleveref}
\usepackage{color,calc,comment}
\usepackage{enumitem}
\usepackage{bbm}
\usepackage{euscript}
\usepackage{mathrsfs}
\usepackage{graphicx}
\usepackage{tkz-euclide, tikz}
\usetikzlibrary{backgrounds,patterns,matrix,calc,arrows,positioning,fit,shapes.geometric,decorations.pathmorphing}
\pgfdeclarelayer{bg}
\pgfdeclarelayer{bg2}    % declare background layer
\pgfsetlayers{bg2,bg,main}
\usepackage{wrapfig}
\setenumerate{labelsep=*, leftmargin=1.5pc}
\newcommand{\RP}{\mathbb{R}\mathrm{P}}

\newcommand{\HC}{\mathbb{C}}

\newcommand{\HN}{\mathbb{N}}

\newcommand{\HR}{\mathbb{R}}

\newcommand{\sC}{\mathscr{C}}

\renewcommand{\d}{\mathrm{ d}}
\renewcommand{\dot}[1]{\overset{\raisebox{-0.25ex}{\scalebox{0.4}{$\bullet$}}}{#1}}

\DeclareMathOperator*{\mean}{\mathbb{E}}

\numberwithin{equation}{section}
\numberwithin{figure}{section}
\theoremstyle{plain}
\newcounter{numbering} \numberwithin{numbering}{section}
\newtheorem{thm}[numbering]{Theorem}
\newtheorem{lemma}[numbering]{Lemma}
\newtheorem{prop}[numbering]{Proposition}

\theoremstyle{definition}

\theoremstyle{remark}

\newtheorem{rem}{Remark}
\newtheorem{example}{Example}
\crefname{equation}{}{}
\crefname{equation}{}{}
\crefname{figure}{Figure}{Figures}
\crefname{section}{Section}{Sections}
\crefname{lemma}{Lemma}{Lemmata}
\crefname{prop}{Proposition}{Propositions}
\crefname{thm}{Theorem}{Theorems}
\crefname{cor}{Corollary}{Corollaries}
\crefname{dfn}{Definition}{Definitions}
\crefname{notation}{Notations}{Notations}
\crefname{rem}{Remark}{Remarks}
\crefname{claim}{Claim}{claims}
\begin{document}
\title{\thetitle}
\author{\theauthor}
\thanks{\hspace{-0.5cm} PB: Max-Planck-Institute for Mathematics in the Sciences (Leipzig), breiding@mis.mpg.de,\\ KK: Max-Planck-Institute for Mathematics in the Sciences (Leipzig), kozhasov@mis.mpg.de,\\ AL: SISSA (Trieste), lerario@sissa.it}

\begin{abstract} We investigate some geometric properties of the real algebraic variety $\Delta$ of symmetric matrices with repeated eigenvalues. We explicitly compute the volume of its intersection with the sphere and prove a Eckart-Young-Mirsky-type theorem for the distance function from a generic matrix to points in $\Delta$.  We exhibit connections of our study to Real Algebraic Geometry (computing the Euclidean Distance Degree of $\Delta$) and Random Matrix Theory.
\end{abstract}
\maketitle
\section{Introduction}
In this paper we investigate the geometry of the set $\Delta$ (below called \emph{discriminant}) of real symmetric matrices with repeated eigenvalues and of unit Frobenius norm:
$$\Delta=\{Q\in \textrm{Sym}(n, \R) :  \lambda_i(Q)=\lambda_j(Q) \textrm{
for some $i\neq j$}\}\cap S^{N-1}.$$
Here, $\lambda_1(Q),\ldots,\lambda_n(Q)$ denote the eigenvalues of $Q$,  the dimension of the space of symmetric matrices is $N :=\frac{n(n+1)}{2}$ and $S^{N-1}$ denotes the unit sphere in $\txt{Sym}(n,\HR)$ endowed with the Frobenius norm $\|Q\| :=\sqrt{\textrm{tr}(Q^2)}$.

This discriminant  is a fundamental object and it appears in several areas of mathematics, from mathematical physics to real algebraic geometry, see for instance \cite{Arnold, Arnold2, Arnold3, Arnold4, rare, Agrachev, AgrachevLerario, Vassiliev}. We discover some new properties of this object (Theorem \ref{thm:vol} and Theorem \ref{thm:EYM}) and exhibit connections and applications of these properties to Random Matrix Theory (Section \ref{sec:RMT}) and Real Algebraic Geometry (Section \ref{sec:ED}).

The set $\Delta$ is an algebraic subset of $S^{N-1}$. It is defined by the \emph{discriminant polynomial}:
\begin{align*}
 \textrm{disc}(Q):=\prod_{i\neq j}(\lambda_i(Q)-\lambda_j(Q))^2,
\end{align*}
which is a non-negative homogeneous polynomial of degree $\txt{deg}(\txt{disc}) =n(n-1)$ in the entries of $Q$. Moreover, it is a sum of squares of real polynomials \cite{Il,Parlett} and $\Delta$ is of codimension two.
The set  $\Delta_\textrm{sm}$ of smooth points of  $\Delta$ is the set of real points of the smooth part of the Zariski closure of $\Delta$ in $\textrm{Sym}(n,\mathbb{C})$ and consists of matrices with \emph{exactly} two repeated eigenvalues. In fact,~$\Delta$ is stratified according to the multiplicity sequence of the eigenvalues; see~\cref{stratification}.
\subsection{The volume of the set of symmetric matrices with repeated eigenvalues}
 Our first main result concerns the computation of the \emph{volume} $|\Delta|$ of the discriminant, which is defined to be the Riemannian volume of the smooth manifold $\Delta_{\txt{sm}}$ endowed with the Riemannian metric induced by the inclusion $\Delta_\textrm{sm}\subset S^{N-1}$.
\begin{thm}[The volume of the discriminant]\label{thm:vol}
\begin{align*}
\frac{|\Delta|}{|S^{N-3}|}={n \choose 2}.
\end{align*}
\end{thm}
\begin{rem}Results of this type (the computation of the volume of some relevant algebraic subsets of the space of matrices) have started appearing in the literature since the 90's \cite{EdelmanKostlan, EdelmanKostlanShub}, with a particular emphasis on asymptotic studies and complexity theory, and have been crucial for the theoretical advance of numerical algebraic geometry, especially for what concerns the estimation of the so called \emph{condition number} of linear problems \cite{Demmel}. The very first result gives the volume of the set $\Sigma\subset \R^{n^2}$ of square matrices with zero determinant and Frobenius norm one; this was computed in \cite{EdelmanKostlan, EdelmanKostlanShub}:
$$ \frac{|\Sigma|}{|S^{n^2-1}|}=\sqrt{\pi}\frac{\Gamma\left(\frac{n+1}{2}\right)}{\Gamma\left(\frac{n}{2}\right)}\sim \sqrt{\frac{\pi}{2}}n^{1/2}.$$
For example, this result is used in  \cite[Theorem 6.1]{EdelmanKostlan} to compute the average number of zeroes of the determinant of a matrix of linear forms.
Subsequently this computation was extended to include the volume of the set of $n\times m$ matrices of given corank in \cite{Beltran} and the volume of the set of \emph{symmetric} matrices with determinant zero in \cite{gap}, with similar expressions. Recently, in \cite{BelKoz} the above formula and \cite[Thm. 3]{gap}  were used to compute the expected condition number of \emph{the polynomial eigenvalue problem} whose input matrices are taken to be random.
\end{rem}
In a related paper \cite{RS} we use Theorem \ref{thm:vol} for counting the average number of singularities of a random spectrahedron. Moreover, the proof of Theorem \ref{thm:vol} requires the evaluation of the expectation of the square of the characteristic polynomial of a GOE$(n)$ matrix (Theorem \ref{thm:charpol} below), which constitutes a result of independent interest.

Theorem \ref{thm:vol} combined with Poincar\'e's kinematic formula from \cite{Howard} allows to compute the \emph{average} number of symmetric matrices with repeated eigenvalues in a uniformly distributed projective two-plane $L \subset \mathrm{P}\textrm{Sym}(n, \R)\simeq\R\txt{P}^{N-1}$:
\begin{equation}\label{eq:meancut} \mean \#(L\cap \mathrm{P}\Delta)=\frac{|\mathrm{P}\Delta|}{|\RP^{N-3}|} =\frac{|\Delta|}{|S^{N-3}|}={n\choose 2},\end{equation}
where by $\txt{P}\Delta\subset \mathrm{P}\textrm{Sym}(n, \R)\simeq\R\txt{P}^{N-1}$ we denote the projectivization of the discriminant.
The following optimal bound on the number $\#(L\cap \txt{P}\Delta)$ of symmetric matrices with repeated eigenvalues in a generic projective two-plane $L\simeq \R\txt{P}^2\subset \R\txt{P}^{N-1}$ was found in \cite[Corollary 15]{SSV}:
\begin{equation}\label{eq:cut} \#(L\cap \mathrm{P}\Delta)\leq {n+1\choose 3}.\end{equation}
\begin{rem} Consequence \eqref{eq:meancut} combined with \eqref{eq:cut} ``violates'' a frequent phenomenon in random algebraic geometry, which goes under the name of \emph{square root law}: for a large class of models of random systems, often related to the so called Edelman-Kostlan-Shub-Smale models \cite{EdelmanKostlan, shsm, EdelmanKostlanShub, Ko2000, ShSm3, ShSm1}, the average number of solutions equals (or is comparable to) the square root of the maximum number; here this is not the case. We also observe that, surprisingly enough, the \emph{average cut} of the discriminant is an integer number (there is no reason to even expect that it should be a rational number!).
\end{rem}
More generally one can ask about the expected number of matrices with a multiple eigenvalue in a ``random'' compact 2-dimensional family. We prove the following.
\begin{thm}[Multiplicities in a random family]\label{prop:family}Let $F:\Omega\to \mathrm{Sym}(n, \R)$ be a random Gaussian field $F=(f_1, \ldots, f_N)$ with i.i.d. components and denote by $\pi:\mathrm{Sym}(n, \R)\backslash\{0\}\to S^{N-1}$ the projection map. Assume that:
  \begin{enumerate}\item with probability one the map $\pi\circ F$ is an embedding and
    \item the expected number of solutions of the random system $\{f_1=f_2=0\}$ is finite.
    \end{enumerate} Then:
$$ \mathbb{E}\#F^{-1}(\sC(\Delta))={n \choose 2}\mean\#\{f_1=f_2=0\},$$
where $\sC(\Delta)\subset \txt{Sym}(n,\R)$ is the cone over $\Delta$.
\end{thm}
\begin{example}
When each $f_i$ is a Kostlan polynomial of degree $d$, then the hypotheses of Theorem~\ref{prop:family} are verified and $\mean \#\{f_1=f_2=0\}=2d |\Omega|/|S^2|$; when each $f_i$ is a degree-one Kostlan polynomial and $\Omega=S^2$, then $\mean\#\{f_1=f_2=0\}=2$ and we recover \eqref{eq:meancut}.
\end{example}

\subsection{An Eckart-Young-Mirsky-type theorem}
The classical Eckart-Young-Mirsky theorem allows to find a best low rank approximation to a given matrix.

For $r\leq m\leq n$ let's denote by $\Sigma_r$ the set of $m\times n$ complex matrices of rank $r$. Then for a given $m\times n$ real or complex matrix $A$ a rank $r$ matrix $\tilde{A}\in \Sigma_r$ which is a global minimizer of the distance function
\begin{equation*}
\txt{dist}_A: \Sigma_r \rightarrow \R,\quad B\mapsto \Vert A-B\Vert:=\sqrt{\sum_{i=1}^m\sum_{j=1}^n |a_{ij}-b_{ij}|^2}
\end{equation*}
is called \emph{a best rank $r$ approximation to $A$}. The Eckart-Young-Mirsky theorem states that if $A=U^*SV$ is the singular value decomposition of $A$, i.e., $U$ is an $m\times m$ real or complex unitary matrix, $S$ is an $m\times n$ rectangular diagonal matrix with non-negative diagonal entries $s_1\geq \dots \geq s_m\geq 0$ and $V$ is an $n\times n$ real or complex unitary matrix, then $\tilde{A} = U^*\tilde{S}V$ is a best rank $r$ approximation to $A$, where $\tilde{S}$ denotes the rectangular diagonal matrix with $\tilde{S}_{ii} = s_i$ for $i=1,\dots,r$ and $\tilde{S}_{jj}=0$ for $j=r+1,\dots,m$.  Moreover, a best rank $r$ approximation to a sufficiently generic matrix is actually unique. More generally, one can show that any critical point of the distance function $\txt{dist}_A:\Sigma_r \rightarrow \R$ is of the form $U^*\tilde{S}^I V,$ where $I\subset \{1,2,\dots,m\}$ is a subset of size $r$ and $\tilde{S}^I$ is the rectangular diagonal matrix with $\tilde{S}^I_{ii} = s_i$ for $i\in I$ and $\tilde{S}^I_{jj} = 0$ for $j\notin I$. In particular, the number of critical points of $\txt{dist}_A$ for a generic matrix $A$ is $\tbinom{n}{r}$. In \cite{EDD} the authors call this count the \emph{Euclidean Distance Degree} of $\Sigma_r$; see also \cref{sec:ED} below.

In the case of real symmetric matrices similar results are obtained by replacing singular values $\sigma_1\geq \dots\geq \sigma_n$ with absolute values of eigenvalues $|\lambda_1|>\dots>|\lambda_n|$ and singular value decomposition $U\Sigma V^*$ with spectral decomposition $C^T\Lambda C$; see \cite[Thm. 2.2]{HS1995} and \cite[Sec. 2]{gap}.

For the distance function from a symmetric matrix to the cone over $\Delta$ we also have an Eckart-Young-Mirsky-type theorem. We prove this theorem in \cref{sec:proof1}.
\begin{thm}[Eckart-Young-Mirsky-type theorem]\label{cor:EYM}
Let $A\in \txt{Sym}(n,\R)$ be a generic real symmetric matrix and let $A=C^T\Lambda C$ be its spectral decomposition with $\Lambda= \txt{diag}(\lambda_1,\dots,\lambda_n)$. Any critical point of the distance function
$$\txt{d}_A: \sC(\Delta_{\txt{sm}})\setminus \{0\} \rightarrow \R$$
is of the form $C^T \Lambda_{i,j} C$, where
\begin{equation*}
\Lambda_{i,j} = \txt{diag}\left(\lambda_1, \dots, \underset{i}{\frac{\lambda_i+\lambda_j}{2}}, \dots, \underset{j}{\frac{\lambda_i+\lambda_j}{2}},\dots,\lambda_n\right),\quad 1\leq i<j\leq n.
\end{equation*}
Moreover, the function $\txt{d}_A: \sC(\Delta) \rightarrow \R$ attains its global minimum at exactly one of the critical points $C^T\Lambda_{i,j}C\in  \sC(\Delta_{\txt{sm}})\setminus \{0\}$ and the value of the minimum of $\txt{d}_A$ equals:
\begin{equation*}
\min\limits_{B\in \sC(\Delta)} \Vert A-B\Vert = \min\limits_{1\leq i<j\leq n} \frac{|\lambda_i-\lambda_j|}{\sqrt{2}}.
\end{equation*}
\end{thm}
\begin{rem}
Since $\sC(\Delta)\subset \txt{Sym}(n,\R)$ is the homogeneous cone over $\Delta\subset S^{N-1}$ the above theorem readily implies an analogous result for the spherical distance function from $A\in S^{N-1}$ to $\Delta$. The critical points are $(1-\frac{(\lambda_i-\lambda_j)^2}{2})^\frac{-1}{2}\,C^T \Lambda_{i,j} C$
and the global minimum of the spherical distance function $\txt{d}^S$ is $\min_{B\in \Delta} \txt{d}^S(A,B) = \min_{1\leq i<j\leq n} \arcsin\left(\tfrac{|\lambda_i-\lambda_j|}{\sqrt{2}}\right)$.
\end{rem}
The theorem is a special case of \cref{thm:EYM} below, that concerns the critical points of the distance points to fixed \emph{stratum} of $\sC(\Delta)$. These strata are in bijection with vectors of natural numbers $w=(w_1,w_2,\dots,w_n)\in \HN^n$ such that $\sum_{i=1}^n i \,w_i = n$ as follows: let us denote by $\sC(\Delta)^w$ the smooth semialgebraic submanifold of $\txt{Sym}(n,\R)$ consisting of symmetric matrices that for each $i\geq 1$ have exactly $w_i$ eigenvalues of multiplicity $i$.
Then, by \cite[Lemma 1]{SV1995}, the semialgebraic sets $\sC(\Delta)^w$ with $w_1<n$ form a stratification of $\sC(\Delta)$:
\begin{equation}\label{stratification}
\sC(\Delta) = \bigsqcup_{w\, :\, w_1<n} \sC(\Delta)^w
\end{equation}
In this notation, the complement of $\sC(\Delta)$ can be written  $\sC(\Delta)^{(n,0,\dots,0)} = \txt{Sym}(n,\R)\setminus \sC(\Delta)$. By \cite[Lemma 1.1]{Arnold}, the codimension of $\sC(\Delta)^w$ in the space $\txt{Sym}(n,\R)$ equals
\begin{equation*}
\txt{codim}(\sC(\Delta)^w) = \sum\limits_{i=1}^n \frac{(i-1)(i+2)}{2}w_i
\end{equation*}
Let us denote by $\txt{Diag}(n,\R)^w:=\txt{Diag}(n,\R)\cap \sC(\Delta)^w$ the set of diagonal matrices in $\sC(\Delta)^w$ and its Euclidean closure by $\overline{\txt{Diag}(n,\R)^w}$. This closure is an arrangement of $\frac{n!}{1!^{w_1}2!^{w_2}3!^{w_3}\dots}$ many $(\sum_{i=1}^n w_i)$-dimensional planes. Furthermore, for a sufficiently generic diagonal matrix $\Lambda=\txt{diag}(\lambda_1,\dots,\lambda_n)$ the distance function
\begin{equation*}
\txt{d}_{\Lambda}: \txt{Diag}(n,\R)^w \rightarrow \R,\quad
\tilde{\Lambda}=\txt{diag}(\tilde{\lambda}_1,\dots,\tilde{\lambda}_n) \mapsto \sqrt{\sum\limits_{i=1}^n (\lambda_i-\tilde{\lambda}_i)^2}
\end{equation*}
has $\frac{n!}{1!^{w_1}2!^{w_2}3!^{w_3}\dots}$ critical points each of which is the orthogonal projection of $\Lambda$ on one of the planes in the arrangement $\overline{\txt{Diag}(n,\R)^w}$ and the distance $\txt{d}_{\Lambda}$ attains its unique global minimum at one of these critical points. We will show that an analogous result holds for
\begin{equation*}
\txt{d}_A: \sC(\Delta)^w \rightarrow \R,\quad
\tilde{A}\mapsto \Vert A-\tilde{A}\Vert=\sqrt{\txt{tr}((A-\tilde{A})^2)},
\end{equation*}
the distance function from a general symmetric matrix $A\in \txt{Sym}(n,\R)$ to the smooth semialgebraic manifold $\sC(\Delta)^w$. The proof for the following theorem is in \cref{sec:proof1}.
\begin{thm}[Eckart-Young-Mirsky-type theorem for the strata]\label{thm:EYM}
Let $A\in \txt{Sym}(n,\R)$ be a generic real symmetric matrix and let $A=C^T\Lambda C$ be its spectral decomposition. Then:
\begin{enumerate}
\item Any critical point of the distance function $\txt{d}_A: \sC(\Delta)^w \rightarrow \R$ is of the form $C^T \tilde{\Lambda} C$, where $\tilde{\Lambda}\in\txt{Diag}(n,\R)^w$ is the orthogonal projection of $\Lambda$ onto one of the planes in $\overline{\txt{Diag}(n,\R)^w}$.
\item The distance function $\txt{d}_A:  \sC(\Delta)^w \rightarrow \R$ has exactly $\frac{n!}{1!^{w_1}2!^{w_2}3!^{w_3}\dots}$ critical points, one of which is the unique global minimum of $\txt{d}_A$.
\end{enumerate}
\end{thm}
\begin{rem}
Note that the manifold $\sC(\Delta)^w$ is not compact and thus the function \mbox{$\txt{d}_A: \sC(\Delta)^w \rightarrow \R$} might not a priori have a minimum.
\end{rem}
\subsection{Euclidean Distance Degree}\label{sec:ED}
Let $X\subset \R^m$ be a real algebraic variety and let $X^\C\subset \C^m$ denote its Zariski closure. The number \mbox{$\#\{x\in X_{\txt{sm}}: u-x\perp T_x X_{\txt{sm}}\}$} of critical points of the distance to the smooth locus $X_{\txt{sm}}$ of $X$ from a generic point $u\in \R^m$ can be estimated by the number $\txt{EDdeg}(X):=\#\{x\in X_{\txt{sm}}^{\C} : u-x \perp T_x X_{\txt{sm}}^{\C}\}$ of ``complex critical points''. Here, $v\perp w$ is orthogonality with respect to the bilinear form $(v,w)\mapsto v^Tw$.  The quantity $\txt{EDdeg}(X)$ does not depend on the choice of the generic point $u\in \R^m$ and it's called \emph{the Euclidean distance degree} of $X$ \cite{EDD}. Also, solutions $x\in X^{\C}_{\txt{sm}}$ to $u-x\perp T_xX_{\txt{sm}}^{\C}$ are called \emph{ED critical points of $u$ with respect to $X$} \cite{DLOT}. In the following theorem we compute the Euclidean distance degree of the variety $\sC(\Delta)\subset \txt{Sym}(n,\R)$ and show that all ED critical points are actually real (this result is an analogue of \cite[Cor. 5.1]{DLOT} for the space of symmetric matrices and the variety $\sC(\Delta)$).
\begin{thm}\label{thm:EDdegree}
Let $A\in \txt{Sym}(n,\R)$ be a sufficiently generic symmetric matrix. Then the ${n \choose 2}$ real critical points of $\txt{d}_A: \sC(\Delta_{\txt{sm}}) \rightarrow \R$ from \cref{cor:EYM} are the only ED critical points of $A$ with respect to $\sC(\Delta)$ and the Euclidean distance degree of $\sC(\Delta)$ equals $\txt{EDdeg}(\sC(\Delta)) = {n \choose 2}$.
\end{thm}
\begin{rem}\label{rem:ED}
An analogous result holds for the Zariski closure of any other stratum of $\sC(\Delta)^w$. Namely, $\txt{EDdeg}(\sc(\Delta)^w)=\frac{n!}{1!^{w_1}2!^{w_2}\dots }$ and for a generic real matrix $A\in \txt{Sym}(n,\HR)$ ED critical points are real and given in Theorem \ref{thm:EYM}.
\end{rem}
\subsection{Random matrix theory}\label{sec:RMT}
The proof of Theorem \ref{thm:vol} eventually arrives at equation \eqref{almost_the_volume2}, which reduces our study to the evaluation of a special integral over the \emph{Gaussian Orthogonal ensemble} ($\mathrm{GOE}$) \cite{mehta,Tao:12}. The connection between the volume of $\Delta$ and random symmetric matrices comes from the fact that, in a sense, the geometry in the Euclidean space of symmetric matrices with the Frobenius norm and the random $\mathrm{GOE}$ matrix model can be seen as the same object under two different points of view.

The integral in \eqref{almost_the_volume2} is the second moment of the characteristic polynomial of a $\mathrm{GOE}$ matrix. In \cite{mehta} Mehta gives a general formula for all moments of the characteristic polynomial of a $\mathrm{GOE}$ matrix. However, we were unable to locate an exact evaluation of the formula for the second moment in the literature. For this reason we added \cref{thm_char_poly}, in which we compute the second moment, to this article. We use it in \cref{sec:RMTproofs} to prove the following theorem.
\begin{thm}\label{thm:charpol}
For a fixed positive integer $k$ we have
  $$\int_{u\in\HR}\,\mean_{Q\sim \mathrm{GOE}(k)} [\det(Q-u\mathbbm{1})^2] \,e^{-u^2}\, \d u = \sqrt{\pi} \,\frac{(k+2)!}{2^{k+1}}.$$
\end{thm}
An interesting remark in this direction is that some geometric properties of $\Delta$ can be stated using the language of Random Matrix Theory. For instance, the estimate on the volume of a tube around $\Delta$ allows to estimate the probability that two eigenvalues of a GOE(n) matrix are close: for $\epsilon>0$ small enough
\begin{equation}\label{eq:close} \mathbb{P}\{\mathrm{min}_{i\neq j}\,|\lambda_i(Q)-\lambda_j(Q)|\leq \epsilon\}\leq \frac{1}{4}{n\choose 2}\epsilon^2+O(\epsilon^3).\end{equation}
The interest of this estimate is that it provides a \emph{non-asymptotic} (as opposed to studies in the limit $n\to \infty$, \cite{tao, BenArous}) result in random matrix theory. It would be interesting to provide an estimate of the implied constant in \eqref{eq:close}, however this might be difficult using our approach as it probably involves estimating higher curvature integrals of $\Delta$.

% %%%%%%%%%%%%%%%%%%%%%%%%%%%%%%%%%%%%%%%%%%%%%%%%%%%%%%%%%%%%%%%%%%%%%%%
% %%%%%%%%%%%%%%%%%%%%%%%%%%%%%%%%%%%%%%%%%%%%%%%%%%%%%%%%%%%%%%%%%%%%%%%
% %%%%%%%%%%%%%%%%%%%%%%%%%%%%%%%%%%%%%%%%%%%%%%%%%%%%%%%%%%%%%%%%%%%%%%%
\section{Critical points of the distance to the discriminant}\label{sec:proof1}

In this section we prove Theorems \ref{cor:EYM}, \ref{thm:EYM} and \ref{thm:EDdegree}. Since Theorem \ref{cor:EYM} is a special case of Theorem~\ref{thm:EYM}, we start by proving the latter.
\subsection{Proof of Theorem \ref{thm:EYM}}
Let's denote by $\overline{\sC(\Delta)^w}\subset \txt{Sym}(n,\R)$ the Euclidean closure of~$\sC(\Delta)^w$. Note that $\overline{\sC(\Delta)^w}$ is a (real) algebraic variety, the smooth locus of $\overline{\sC(\Delta)^w}$ is $\sC(\Delta)^w$ and the boundary $\overline{\sC(\Delta)^w}\setminus \sC(\Delta)^w$ is a union of some strata $\sC(\Delta)^{w^\prime}$ of greater codimension.

Let now $A\in \txt{Sym}(n,\R)$ be a sufficiently generic symmetric matrix and let $A=C^T\Lambda C$ be its spectral decomposition.
From \cite[Thm. 3]{BD2017} and \cite[Sec. 3]{BD2017} it follows that any real ED critical point of $\overline{\sC(\Delta)^w}$ with respect to $A=C^T\Lambda C$ is of the form $C^T\tilde{\Lambda}C$, where the diagonal matrix $\tilde{\Lambda}\in \txt{Diag}(n,\R)^w$ is a ED critical point of $\overline{\txt{Diag}(n,\R)^w}$ with respect to $\Lambda \in \txt{Diag}(n,\R)$.
Since, as observed above $\overline{\txt{Diag}(n,\R)^w}$, is an arrangement of $\frac{n!}{1!^{w_1}2!^{w_2}3!^{w_3}\dots}$ planes its ED critical points with respect to a generic $\Lambda\in \txt{Diag}(n,\R)$ are the orthogonal projections of $\Lambda$ on the components of the plane arrangement. One of these ED critical points is the (unique) closest point on $\overline{\txt{Diag}(n,\R)^w}$ to the generic $\Lambda\in \txt{Diag}(n,\R)$. Both claims follow. \qed
\subsection{Proof of Theorem \ref{cor:EYM}}
Let $w=(n-2,1,0,\dots,0)$ and let's for a given symmetric matrix $A\in \txt{Sym}(n,\HR)$ fix a spectral decomposition $A=C^T\Lambda C, \Lambda=\txt{diag}(\lambda_1,\dots,\lambda_n)$. From Theorem~\ref{thm:EYM} we know that the critical points of the distance function $\txt{d}_A: \sC(\Delta_{\txt{sm}})\setminus \{0\} \rightarrow \R$ are of the form $C^T\Lambda_{i,j}C$, $1\leq i<j\leq n$,
where $\Lambda_{i,j}$  is the orthogonal projection of $\Lambda$ onto the hyperplane $\{\lambda_i=\lambda_j\}\subset \overline{\txt{Diag}(n,\R)^w}$. It is straightforward to check that
$$\Lambda_{i,j} = \txt{diag}\left(\lambda_1, \dots, \underset{i}{\frac{\lambda_i+\lambda_j}{2}}, \dots, \underset{j}{\frac{\lambda_i+\lambda_j}{2}},\dots,\lambda_n\right).$$
From this, it is immediate that the distance between $\Lambda$ and $\Lambda_{i,j}$ equals
\begin{equation*}
\Vert \Lambda-\Lambda_{i,j}\Vert = \sqrt{\txt{tr}\left(\left(\Lambda-\Lambda_{i,j}\right)^2\right)} = \frac{|\lambda_i-\lambda_j|}{\sqrt{2}}
\end{equation*}
This finishes the proof.
\qed

\subsection{Proof of Theorem \ref{thm:EDdegree}}
In the proof of \cref{cor:EYM} we showed that there are $\tbinom{n}{2}$ real ED critical points of the distance function from a general real symmetric matrix $A$ to $\sC(\Delta)$. In this subsection we in particular argue that there are no other (complex) ED critical points in this case.

Let $X^{\HC}\subset \txt{Sym}(n,\HC)$ be the Zariski closure of $\sC(\Delta)\subset \txt{Sym}(n,\HR)$. The orthogonal group $G=O(n)=\{C\in M(n,\HR) : C^TC=\mathbbm{1}\}$ and the complex orthogonal group $G^{\HC}=\{C\in M(n,\HC) : C^T C= \mathbbm{1}\}$ act by conjugation on $\txt{Sym}(n,\HR)$ and on $\txt{Sym}(n,\HC)$ respectively.
Since $\sC(\Delta)\subset \txt{Sym}(n,\HR)$ is $G$-invariant, by \cite[Lemma 2.1]{DLOT}, the complex variety $X^{\HC}\subset \txt{Sym}(n,\HC)$ is also $G$-invariant. Using the same argument as in \cite[Thm. 2.2]{DLOT} we now show that $X^{\HC}$ is actually $G^{\HC}$-invariant. Indeed, for a fixed point $A\in X^{\HC}$ the map
\begin{equation*}
\gamma_A: G^{\HC} \rightarrow \txt{Sym}(n,\HC),\quad C \mapsto C^T A C
\end{equation*}
is continuous and hence the set $\gamma_A^{-1}(X^{\HC})\subset G^{\HC}$ is closed. Since by the above $G\subset \gamma_A^{-1}(X^{\HC})$ and since $G^{\HC}$ is the Zariski closure of $G$ we must have $\gamma_A^{-1}(X^{\HC}) = G^{\HC}$.

Let's denote by $\txt{Diag}(n,\HC)\subset \txt{Sym}(n,\HC)$ the space of complex diagonal matrices. For any matrix $D\in \txt{Diag}(n,\HC)$ with pairwise distinct diagonal entries the tangent space at $D$ to the orbit~$G^{\HC}D = \{CD: C\in G^{\HC}\}$ consists of complex symmetric matrices with zeros on the diagonal:
\begin{equation*}
T_D(G^{\HC} D) = \{ vD-Dv: v^T+v=0\} = \{A\in \txt{Sym}(n,\HC): a_{11} =\dots = a_{nn}= 0\}
\end{equation*}
In particular,
\begin{equation*}
T_A\txt{Sym}(n,\HC) = T_A\txt{Diag}(n,\HC)+T_D(G^{\HC}D)
\end{equation*}
is the direct sum which is orthogonal with respect to the bilinear form $(A_1,A_2)\mapsto\txt{tr}(A_1^TA_2)$.

As any real symmetric matrix can be diagonalized by some orthogonal matrix we have $$\sC(\Delta)=G(\sC(\Delta)\cap \txt{Diag}(n,\HR)) = \{CA: C\in G, A\in \sC(\Delta)\cap \txt{Diag}(n,\HR)\}.$$
This, together with the inclusion $\sC(\Delta) \subset G^{\HC}(X^{\HC}\cap \txt{Diag}(n,\HC))$, imply that  \mbox{$G^{\HC}(X^{\HC} \cap \txt{Diag}(n,\HC))$} is Zariski dense in $X^{\HC}$. We now apply the main theorem from \cite{BD2017} to obtain that the ED degree of $X^{\HC}$ in $\txt{Sym}(n,\HC)$ equals the ED degree of \mbox{$X^{\HC}\cap \txt{Diag}(n,\HC)$} in $\txt{Diag}(n,\HC)$. Since $X^{\HC}\cap \txt{Diag}(n,\HC)= \{D\in \txt{Diag}(n,\HC): D_i=D_j, i\neq j\}$ is the union of ${n \choose 2}$ hyperplanes the ED critical points of a generic $\tilde{D} \in \txt{Diag}(n,\HC)$ are orthogonal projections from $\tilde{D}$ to each of the hyperplanes (as in the proof of Theorem \ref{cor:EYM}). In particular $\txt{EDdeg}(X^{\HC}) = \txt{EDdeg}(X^{\HC}\cap \txt{Diag}(n,\HC)) = {n\choose 2}$ and if $\tilde D\in \txt{Diag}(n,\HR)$ is a generic real diagonal matrix ED critical points are all real. Finally, for a general symmetric matrix $A=C^T\tilde D C$ all ED critical points are obtained from the ones for $\tilde{D}\in \txt{Diag}(n,\HR)$ via conjugation by $C\in O(n)$.

The proof of the statement in \cref{rem:ED} is similar. Each plane in the plane arrangement $\overline{\txt{Diag}(n,\R)^w}$ yields one critical point and there are  $\frac{n!}{1!^{w_1}2!^{w_2}3!^{w_3}\dots}$ many such planes.
\qed

\section{The volume of the discriminant}
The goal of this section is to prove Theorem \ref{thm:vol} and Theorem \ref{prop:family}. As was mentioned in the introduction, we reduce the computation of the volume to an integral over the $\mathrm{GOE}$-ensemble. This is why, before starting the proof, in the next subsection we recall some preliminary concepts and facts from random matrix theory that will be used in the sequel.
\subsection{The $\textrm{GOE}(n)$ model for random matrices }
The material we present here is from \cite{mehta}.

The \emph{$\txt{GOE}(n)$ probability measure} of any Lebesgue measurable subset $U\subset \txt{Sym}(n,\R)$ is defined as follows:
\begin{align*}
  \mathbb{P}\{U\} = \frac{1}{\sqrt{2}^n\sqrt{\pi}^{N}} \int_U e^{-\frac{\Vert A\Vert^2}{2}}\, dA,
\end{align*}
where $dA=\prod_{1\leq i\leq j\leq n} dA_{ij}$ is the Lebesgue measure on the space of symmetric matrices $\txt{Sym}(n,\R)$ and, as before, $\Vert A\Vert = \sqrt{ \txt{tr}(A^2)}$ is the Frobenius norm.

By \cite[Sec. 3.1]{mehta}, the joint density of the eigenvalues of a $\mathrm{GOE}(n)$ matrix $A$ is given by the measure $
  \tfrac{1}{Z_n} \int_V e^{-\frac{\Vert \lambda\Vert^2}{2}} |\Delta(\lambda)|\, d\lambda,$
where $d\lambda=\prod_{i=1}^n d\lambda_i$ is the Lebesgue measure on $\R^n$, $V\subset \R^n$ is a measurable subset, $\Vert \lambda\Vert^2 = \lambda_1^2+\dots+\lambda_n^2$ is the Euclidean norm, $\Delta(\lambda):=\prod_{1\leq i<j\leq n}(\lambda_j-\lambda_i)$ is the Vandermonde determinant and $Z_n$ is the normalization constant whose value is given by the formula
\begin{align}\label{selberg}
  Z_n = \int_{\R^n}  e^{-\frac{\Vert\lambda\Vert^2}{2}} |\Delta(\lambda)| \,\d\lambda =\sqrt{2\pi}^{\,n} \prod_{i=1}^n \frac{\Gamma(1+\tfrac{i}{2})}{\Gamma(\tfrac{3}{2})},
\end{align}
see \cite[Eq. (17.6.7)]{mehta} with $\gamma=a=\tfrac{1}{2}$. In particular, for an integrable function $f:\txt{Sym}(n,\R) \to \R$ that depends only on the eigenvalues of $A\in \txt{Sym}(n,\R)$, the following identity holds
\begin{equation}\label{eigenvalue_density}
\mean_{A\sim \mathrm{GOE}(n)} f(A) = \frac{1}{Z_n} \int_V f(\lambda_1,\ldots,\lambda_n)\,e^{-\frac{\Vert \lambda\Vert^2}{2}} |\Delta(\lambda)|\, d\lambda.
\end{equation}

% The unit sphere $S^{N-1}:=\{Q\in \txt{Sym}(n,\R): \Vert Q\Vert=1\}$ in the space of symmetric matrices is endowed with the uniform measure. Then the normalized volume $|E|/|S^{N-1}|$ of any measurable set $E\subset S^{N-1}$ equals the GOE$(n)$ measure of the cone $C(E):=\{Q\in \txt{Sym}(n,\R) :  Q/\Vert Q\Vert\in E\}$ over $E$:
% \begin{align}\label{relation}
% \frac{|E|}{|S^{N-1}|} = \mathbb{P}\{C(E)\}
% \end{align}
% We will use this property of the GOE measure when needed.

\subsection{Proof of Theorem \ref{thm:vol}}

In what follows we endow the orthogonal group $O(n)$ with the left-invariant metric defined on the Lie algebra $T_{\mathbbm{1}}O(n)$ by
\begin{equation*}
\langle u,v\rangle = \frac{1}{2}\txt{tr}(u^Tv),\ u,v \in T_{\mathbbm{1}}O(n)
\end{equation*}
The following formula for the volume of $O(n)$ can be found in \cite[Corollary 2.1.16]{Muirhead}:
\begin{align}\label{eq:volume of O(n)}
\vert O(n) \vert = \frac{2^n \pi^\frac{n(n+1)}{4}}{\prod_{i=1}^{n} \Gamma\left(\frac{i}{2}\right)}
\end{align}
Recall that by definition the volume of $\Delta$ equals the volume of the smooth part $\Delta_\mathrm{sm}\subset S^{N-1}$ that consists of symmetric matrices of unit norm with exactly two repeated eigenvalues. Let's denote by $(S^{n-2})_*$ the dense open subset of the $(n-2)$-sphere consisting of points with pairwise distinct coordinates.  We consider the following parametrization of $\Delta_{\txt{sm}}\subset S^{N-1}$:
\begin{equation*}
p : O(n) \times (S^{n-2})_* \rightarrow \Delta_\mathrm{sm},\quad (C, \mu)\mapsto C^T \,\mathrm{diag}(\lambda_1,\ldots,\lambda_n)\,C,
\end{equation*}
where $\lambda_1,\ldots,\lambda_n$ are defined as
\begin{equation}\label{parametrizing_eigenvalues}
  \lambda_1 = \mu_1,\;  \lambda_2 = \mu_2,\;\ldots, \;\lambda_{n-2} = \mu_{n-2}, \;\lambda_{n-1} = \frac{\mu_{n-1}}{\sqrt{2}},\;\lambda_{n} = \frac{\mu_{n-1}}{\sqrt{2}}.
\end{equation}
In Lemma \ref{lem_normal_jacobian} below we show that $p$ is a submersion. Applying to it the smooth coarea formula (see, e.g., \cite[Theorem 17.8]{BuCu}) we have
\begin{equation}\label{almost_the_volume}
    \int_{A\in\Delta_{\txt{sm}}}\,\vert p^{-1}(A)\vert\,\d A
    =
  \int_{(C,\mu)\in O(n)\times (S^{n-2})_*}\,\mathrm{NJ}_{(C,\mu)}p\, \d (C,\mu)
\end{equation}
Here $\mathrm{NJ}_{(C,\mu)}p$ denotes the \emph{normal Jacobian} of $p$ at $(C,\mu)$ and we compute its value in the following lemma.
\begin{lemma}\label{lem_normal_jacobian}
The parametrization $p: O(n)\times (S^{n-2})_* \rightarrow \Delta_{\txt{sm}}$ is a submersion and its normal Jacobian at $(C,\mu)\in O(n)\times (S^{n-2})_*$ is given by the formula
 $$\mathrm{NJ}_{(C,\mu)}p = \sqrt{2}^{\frac{n(n-1)}{2}-1}\prod_{1\leq i<j\leq n-2} \vert\mu_i-\mu_j\vert\, \prod_{i=1}^{n-2}\left\vert\mu_i-\frac{\mu_{n-1}}{\sqrt{2}}\right\vert^2.$$
\end{lemma}
\begin{proof}
Recall that for a smooth submersion $f: M\rightarrow N$ between two Riemannian manifolds the normal Jacobian of $f$ at $x\in M$ is the absolute value of the determinant of the restriction of the differential $D_{x}f: T_x M \rightarrow T_{f(x)}N$ of $f$ at $x$ to the orthogonal complement of its kernel. We now show that the parametrization $p: O(n)\times (S^{n-2})_* \rightarrow \Delta_{\txt{sm}}$ is a submersion and compute its normal Jacobian.

Note that  $p$ is equivariant with respect to the right action of $O(n)$ on itself and its action on $\Delta_{\txt{sm}}$ via conjugation, i.e., for all $C, \tilde{C}\in O(n)$ and $\mu\in S^{n-2}$ we have $p(C\tilde{C},\mu) = \tilde{C}^Tp(C,\mu)\tilde{C}^T$. Therefore, $D_{(C,\mu)}p = C^TD_{(\mathbbm{1},\mu)}p\, C$ and, consequently, $\mathrm{NJ}_{(C,\mu)}p = \mathrm{NJ}_{(\mathbbm{1},\mu)}p$. We compute the latter.
The differential of $p$ at~$(I,\mu)$ is the map
  \begin{align*}
    D_{(\mathbbm{1},\mu)}p: T_{\mathbbm{1}} O(n) \times T_\mu S^{n-2} &\to T_{p(\mathbbm{1},\mu)} \Delta_\mathrm{sm},\\ (\dot{C},\dot{\mu}) &\mapsto \dot{C}^T \mathrm{diag}(\lambda_1,\ldots,\lambda_n) + \mathrm{diag}(\lambda_1,\ldots,\lambda_n)\,\dot{C} + \mathrm{diag}(\dot{\lambda}_1,\ldots,\dot{\lambda}_n),\nonumber
  \end{align*}
where $\dot{\lambda}_i = \dot{\mu}_i$ for $1\leq i\leq n-2$ and $\dot{\lambda}_{n-1} = \dot{\lambda}_n = \tfrac{\dot{\mu}_{n-1}}{\sqrt{2}}$. The Lie algebra $T_{\mathbbm{1}}O(n)$ consists of skew-symmetric matrices: $$T_{\mathbbm{1}} O(n) = \{\dot{C}\in \mathbb{R}^{n\times n} : \dot{C}^T = -\dot{C}\}.$$
Let $E_{i,j}$ be the matrix that has zeros everywhere except for the entry $(i,j)$ where it equals $1$. Then $\{E_{i,j} - E_{j,i} : 1\leq i < j\leq n\}$ is an orthonormal basis for $T_{\mathbbm{1}}O(n)$. One verifies that
  \begin{align*}
    D_{(\mathbbm{1},\mu)}p(E_{i,j} - E_{j,i},0) &= (\lambda_j-\lambda_i)(E_{i,j} + E_{j,i}),\quad\text{and}\quad
   D_{(\mathbbm{1},\mu)}p(0, \dot{\mu}) = \txt{diag}(\dot{\lambda}_1,\dots,\dot{\lambda}_{n}).
\end{align*}
This implies that $p$ is a submersion and
\begin{equation*}
   (\ker D_{(\mathbbm{1},\mu)}p)^{\perp} = \mathrm{span}\{E_{i,j} - E_{j,i} : 1\leq i<j\leq n, (i,j)\neq (n-1,n)\}\oplus^{\perp} T_{\mu}(S^{n-2})_*.
\end{equation*}
Combining this with the fact that the restriction of $D_{(\mathbbm{1},\mu)}p$ to $T_{\mu}(S^{n-2})_*$ is an isometry we obtain
\begin{align*}
\txt{NJ}_{(\mathbbm{1},\mu)}p &= \sqrt{2}^{\frac{n(n-1)}{2}-1}\prod\limits_{1\leq i<j\leq n,\,(i,j)\neq (n,n-1)} |\lambda_i-\lambda_j|\\ &= \sqrt{2}^{\frac{n(n-1)}{2}-1} \prod_{1\leq i<j \leq n-2} \vert \mu_i-\mu_j\vert \, \prod_{i=1}^{n-2}\left\vert \mu_i - \frac{\mu_{n-1}}{\sqrt{2}}\right\vert^2,
\end{align*}
which finishes the proof.
\end{proof}
We now compute the volume of the fiber $p^{-1}(A), A\in \Delta_{\txt{sm}}$ that appears in \eqref{almost_the_volume}.
\begin{lemma}\label{lem volume of the fiber}
The volume of the fiber of $A\in \Delta_{\txt{sm}}$ under $p$ equals
$\vert p^{-1}(A)\vert = 2^{n}\pi\,(n-2)! .$
\end{lemma}
\begin{proof}
Let $A= p(C,\mu) \in \Delta_\mathrm{sm}$. The last coordinate $\mu_{n-1}$ is always mapped to the double eigenvalue $\lambda_{n-1}=\lambda_n$ of $A$, whereas there are $(n-2)!$ possibilities to arrange $\mu_1,\dots,\mu_{n-2}$. For fixed choice of $\mu$ there are $|O(1)|^{n-2}|O(2)|$ ways to choose $C\in O(n)$. Therefore, invoking \eqref{eq:volume of O(n)}, we obtain
$ \vert p^{-1}(A)\vert = \vert O(1)\vert^{n-2} \vert O(2)\vert \,(n-2)!=2^{n-2}\cdot 2^2\pi\cdot (n-2)! = 2^{n} \pi (n-2)!$.
\end{proof}
Combining \eqref{almost_the_volume} with Lemma \ref{lem_normal_jacobian} and Lemma \ref{lem volume of the fiber} we write for the normalized volume of $\Delta$:
\begin{equation*}
\frac{\vert \Delta\vert}{\vert S^{N-3}\vert} = \frac{\sqrt{2}^{\frac{n(n-1)}{2}-1-2n}}{\pi(n-2)!\vert S^{N-3}\vert}\int_{(C,\mu)\in O(n)\times  (S^{n-2})_*}\vert\Delta(\mu_1,\ldots,\mu_{n-2})\vert\,\prod_{i=1}^{n-2}\left\vert\mu_i-\frac{\mu_{n-1}}{\sqrt{2}}\right\vert^2\, \d (C,\mu),
\end{equation*}
where $\Delta(\mu_1,\dots,\mu_{n-2})=\prod_{1\leq i<j\leq n-2}(\mu_i-\mu_j)$.

The function being integrated is independent of $C\in O(n)$. Thus, using Fubini's theorem we can perform the integration over the orthogonal group. Furthermore, the integrand is a homogeneous function of degree $\tfrac{(n-2)(n+1)}{2}$. Passing from spherical coordinates to spatial coordinates and extending the domain of integration to the measure-zero set of points with repeated coordinates we obtain
$$\frac{\vert \Delta\vert}{\vert S^{N-3}\vert} = \frac{\sqrt{2}^{\frac{n(n-1)}{2}-1-2n}\vert O(n)\vert} {\pi\,(n-2)!\,\vert S^{N-3}\vert\,K}\,\int_{\mu \in \mathbb{R}^{n-1}}\,\vert\Delta(\mu_1,\ldots,\mu_{n-2})\vert\,\prod_{i=1}^{n-2}\left\vert\mu_i-\frac{\mu_{n-1}}{\sqrt{2}}\right\vert^2\, e^{-\frac{\Vert \mu \Vert^2}{2}}\d\mu$$
where
$$K=\int_0^\infty r^{\frac{(n-2)(n+1)}{2}+n-2} e^\frac{-r^2}{2}\,\d r = 2^{\frac{n(n+1)}{4}-2}\,\Gamma\left(\frac{n(n+1)}{4}-1 \right).$$
Let us write $u:=\tfrac{\mu_{n-1}}{\sqrt{2}}$ for the double eigenvalue and make a change of variables from $\mu_{n-1}$ to~$u$. Considering the eigenvalues $\mu_1,\ldots,\mu_{n-2}$ as the eigenvalues of a symmetric $(n-2)\times(n-2)$ matrix $Q$, by (\ref{eigenvalue_density}) we have
\begin{equation}\label{almost_the_volume1}
\frac{\vert \Delta\vert}{\vert S^{N-3}\vert} = \frac{\sqrt{2}^{\frac{n(n-1)}{2}-2n} \vert O(n) \vert\,Z_{n-2}} {\pi\,(n-2)!\,\vert S^{N-3}\vert\,K}\,\int_{u\in\mathbb{R}}\,\mean_{Q\sim \mathrm{GOE}(n-2)} [ \det(Q-u\mathbbm{1})^2] \, e^{-u^2}\d u.
\end{equation}
Using formulas \eqref{selberg} and \eqref{eq:volume of O(n)} for $Z_{n-2}$ and $\vert O(n)\vert$ respectively we write
  \begin{align*}
   \vert O(n)\vert\cdot Z_{n-2} & =     \frac{2^{n} \pi^\frac{n(n+1)}{4}}{\prod_{i=1}^{n} \Gamma(\tfrac{i}{2})} \cdot \sqrt{2\pi}^{\,n-2} \prod_{i=1}^{n-2} \frac{\Gamma(1+\tfrac{i}{2})}{\Gamma(\tfrac{3}{2})}\\
    & = \frac{2^{n} \pi^\frac{n(n+1)}{4}}{\prod_{i=1}^{n} \Gamma(\tfrac{i}{2})} \cdot \sqrt{2\pi}^{\,n-2} \, \frac{\prod_{i=1}^{n-2} \frac{i}{2}\,\Gamma(\tfrac{i}{2})}{(\frac{\sqrt{\pi}}{2})^{n-2}}\\
    & = \frac{\sqrt{2}^{3n-2} \pi^\frac{n(n+1)}{4}\,(n-2)!}{\Gamma(\frac{n}{2})\Gamma(\frac{n-1}{2})}\\
    & = \sqrt{2}^{5n-6} \pi^{\frac{n(n+1)}{4}-1},
  \end{align*}
where in the last step the duplication formula for Gamma function $\Gamma(\tfrac{n}{2})\Gamma(\tfrac{n-1}{2}) = 2^{2-n} \sqrt{\pi}\,(n-2)!$ has been used. Let's recall the formula for the volume of the $(N-3)$-dimensional unit sphere:  $|S^{N-3}| = 2\pi^\frac{N-2}{2} /\Gamma(\tfrac{N-2}{2})$. Recalling that $N=\frac{n(n+1)}{2}$ we simplify the constant in~(\ref{almost_the_volume1})
\begin{align*}
\frac{\sqrt{2}^{\frac{n(n-1)}{2}-2n}}{\pi\,(n-2)!}  \cdot \frac{\vert O(n)\vert\cdot Z_{n-2}}{\vert S^{N-3}\vert\cdot K}
 = \;&\frac{\sqrt{2}^{\frac{n(n-1)}{2}-2n}}{\pi\,(n-2)!} \cdot  \sqrt{2}^{5n-6} \pi^{\frac{n(n+1)}{4}-1} \cdot \frac{\Gamma(\frac{N-2}{2})}{2\pi^\frac{N-2}{2}}\cdot \frac{2^{2-\frac{n(n+1)}{4}}}{\Gamma(\frac{n(n+1)}{4}-1 )}\\
= \;&\frac{2^{n-2}}{\sqrt{\pi}\,(n-2)!}.
\end{align*}
Plugging this into (\ref{almost_the_volume1}) we have
\begin{equation}\label{almost_the_volume2}
\frac{\vert \Delta\vert}{\vert S^{N-3}\vert} = \frac{2^{n-1}}{\sqrt{\pi}\,n!} \,\binom{n}{2}\,\int_{u\in\mathbb{R}}\,\mean_{Q\sim \mathrm{GOE}(n-2)} [\det(Q-u\mathbbm{1})^2]\, e^{-u^2}\d u.
\end{equation}
Combining the last formula with Theorem \ref{thm:charpol} whose proof is given in Section \ref{sec:RMTproofs} we finally derive the claim of Theorem \ref{thm:vol}: $
\tfrac{|\Delta|}{|S^{N-3}|} = {n \choose 2}.$
\begin{rem}
The proof can be generalized to subsets of $\Delta$ that are defined by an eigenvalue configuration given by a measurable subset of $(S^{n-2})_*$. Such a configuration only adjust the domain of integration in \cref{almost_the_volume2}. For instance, consider the subset
$$(S^{n-2})_1 = \{(\mu_1,\ldots,\mu_{n-1}) \in (S^{n-2})_*\mid \mu_{n-1} < \mu_i \text{ for } 1\leq i\leq n-2\}.$$
It is an open semialgebraic subset of $(S^{n-2})_*$ and $\Delta_1:=p(O(n) \times (S^{n-2})_1)$ is the smooth part of the matrices whose two smallest eigenvalues coincide. Following the proof until \cref{almost_the_volume2}, we get
\begin{equation*}
\frac{\vert \Delta_1\vert}{\vert S^{N-3}\vert} = \frac{2^{n-1}}{\sqrt{\pi}\,n!} \,\binom{n}{2}\,\int_{u\in\mathbb{R}}\,\mean_{Q\sim \mathrm{GOE}(n-2)} [\det(Q-u\mathbbm{1})^2 \, \mathbf{1}_{\{Q \succ u\mathbbm{1}\}}]\, e^{-u^2}\d u,
\end{equation*}
where $\mathbf{1}_{\{Q \succ u\mathbbm{1}\}}$ is the indicator function of $Q-u\mathbbm{1}$ being positive definite.
\end{rem}
\subsection{Multiplicities in a random family}
In this subsection we prove Theorem \ref{prop:family}.

The proof consists of an application of the Integral Geometry Formula from \cite[p.\ 17]{Howard}. Denote $\widehat{F}=\pi\circ F:\Omega\to S^{N-1}$. Then, by assumption, with probability one we have:
\begin{equation*}
  \#F^{-1}(\sC(\Delta))=\#\widehat{F}^{-1}(\Delta)=\#\widehat F(\Omega)\cap \Delta.
\end{equation*}
Observe also that, since the list $(f_1, \ldots, f_N)$ consists of i.i.d. random Gaussian fields, then for every $g\in O(N)$ the random maps $\widehat F$ and $g\circ \widehat F$ have the same distribution and
\begin{align*}\mean  \#F^{-1}(\sC(\Delta))&=\mean\#\widehat F(\Omega)\cap \Delta\\
&=\mean \#g(\widehat F(\Omega))\cap \Delta\\
&=\frac{1}{|O(N)|}\int_{O(N)}\mean \#(g(\widehat F(\Omega))\cap \Delta) dg\\
&=\mean \frac{1}{|O(N)|}\int_{O(N)} \#(g(\widehat F(\Omega))\cap \Delta )dg\\
&=\mean 2\frac{|\widehat F(\Omega)|}{|S^2|}\frac{| \Delta|}{|S^{N-3}|}={n\choose 2}\mean 2\frac{|\widehat F(\Omega)|}{|S^2|}.
\end{align*}
We have used the Integral Geometry Formula \cite[p.\ 17]{Howard} in the last step. Let $L=\{x_1=x_2=0\}$ be the codimension-two subspace of $\mathrm{Sym}(n, \R)$ given by the vanishing of the first two coordinates (in fact: any two coordinates). The conclusion follows by applying integral geometry again:
\begin{equation*} \mean 2\frac{|\widehat F(\Omega)|}{|S^2|}=\mean \frac{1}{|O(N)|}\int_{O(N)} \#(g(\widehat F(\Omega))\cap L )dg=\mathbb{E}\#F^{-1}(L)=\mean \#\{f_1=f_2=0\}.\end{equation*}
This finishes the proof.\qed

\section{The second moment of the characteristic polynomial of a goe matrix}\label{sec:RMTproofs}
In this section we give a proof of Theorem \ref{thm:charpol}. Let us first recall some ingredients and prove some auxiliary results.
\begin{lemma}\label{ratio} Let $P_m=2^{1-m^2}\sqrt{\pi}^{m}\prod_{i=0}^m (2i)!$ and let $Z_{2m}$ be the normalization constant from~\cref{selberg}. Then
  $P_m=2^{1-2m}\,Z_{2m}.$
\end{lemma}
\begin{proof}
The formula \cref{selberg} for $Z_{2m}$ reads
$$Z_{2m} = \sqrt{2\pi}^{2m} \prod_{i=1}^{2m} \frac{\Gamma\left(\frac{i}{2}+1\right)}{\Gamma\left(\frac{3}{2}\right)} = (2\pi)^m \prod_{i=1}^{m} \frac{\Gamma\left(\tfrac{2i-1}{2} +1\right)\Gamma\left(\tfrac{2i}{2} +1\right)}{\left(\frac{\sqrt{\pi}}{2}\right)^2}.$$
Using the formula  $\Gamma(z)\Gamma(z+\tfrac{1}{2}) = \sqrt{\pi}2^{1-2z} \Gamma(2z)$ \cite[43:5:7]{atlas} with $z=i+1/2$ we obtain
$$Z_{2m}= 2^{\,3m} \prod_{i=1}^m \sqrt{\pi} 2^{1-2(i+1/2)} \Gamma(2(i+1/2))=2^{\,2m-m^2}\sqrt{\pi}^m \prod_{i=1}^m (2i)! = 2^{\,2m-1}P_m.$$
This proves the claim.
\end{proof}

Recall now that the \emph{(physicist's) Hermite polynomials} $H_i(x),\, i=0,1,2, \dots$ form a family of orthogonal polynomials on the real line with respect to the measure $e^{-x^2}dx$. They are defined by
\begin{align*}
  H_i(x)=(-1)^i e^{x^2}\frac{\textrm{d}^i}{\textrm{d}x^i} e^{-x^2},\quad i\geq 0
\end{align*}
and satisfy
\begin{equation}\label{ortho_rel}
\int_{u\in\HR}  H_i(u)H_j(u)e^{-u^2}\,\d u = \begin{cases} 2^ii!\sqrt{\pi},& \text{ if } i=j\\ 0,&\text{ else.}\end{cases}
\end{equation}
A Hermite polynomial is either odd (if the degree is odd) or even (if the degree is even) function:
\begin{equation}\label{hermite_symmetry}
H_i(-x)=(-1)^i H_{i}(x);
\end{equation}
and its derivative satisfies
\begin{equation}\label{hermite_derivative}
H_i'(x)=2iH_{i-1}(x);
\end{equation}
see \cite[(24:5:1)]{atlas}, \cite[(8.952.1)]{Gradshteyn2015} for these properties.

The following proposition is crucial for the proof of \cref{thm:charpol}.
\begin{prop}[Second moment of the characteristic polynomial]
\label{thm_char_poly} For a fixed positive integer $k$ and a fixed $u\in\HR$ the following holds.
\begin{enumerate} \item If $k=2m$ is even, then
  $$\mean\limits_{Q\sim \mathrm{GOE}(k)}\det(Q-u\mathbbm{1})^2 = \frac{(2m)!}{2^{2m}}\,\sum_{j=0}^m \frac{2^{-2j-1}}{(2j)!}\,
  \det X_j(u),$$
  where $$X_j(u)=\begin{pmatrix}  H_{2j}(u) &  H_{2j}'(u) \\ H_{2j+1}(u)-H_{2j}'(u) & H_{2j+1}'(u)-H_{2j}''(u) \end{pmatrix}.$$
\item If $k=2m+1$ is odd, then
$$\mean\limits_{Q\sim \mathrm{GOE}(k)}\det(Q-u\mathbbm{1})^2 =\frac{\sqrt{\pi}(2m+1)!}{2^{4m+2}\,\Gamma(m+\tfrac{3}{2})} \sum_{j=0}^m \frac{2^{-2j-2}}{(2j)!}\,
\det Y_j(u),$$
where
$$Y_j(u)=\begin{pmatrix}
 \frac{(2j)!}{j!} & H_{2j}(u) &  H_{2j}'(u) \\
0 & H_{2j+1}(u)-H_{2j}'(u) & H_{2j+1}'(u)-H_{2j}''(u)\\
 \tfrac{(2m+2)!}{(m+1)!} & H_{2m+2}(u) &  H_{2m+2}'(u)
\end{pmatrix}.$$
\end{enumerate}
\end{prop}
\begin{proof}
In Section $22$ of \cite{mehta} one finds two different formulas for the even $k=2m$ and odd $k=2m+1$ cases. We evalute both seperately.

If $k=2m$, we have by \cite[(22.2.38)]{mehta} that
  $$\mean\det(Q-u\mathbbm{1})^2 =\frac{(2m)!\,P_m}{Z_{2m}} \sum_{j=0}^m \frac{2^{{2j-1}}}{(2j)!}\,
  \det\begin{pmatrix} R_{2j}(u)& R_{2j}'(u) \\ R_{2j+1}(u)& R_{2j+1}'(u) \end{pmatrix},$$
where $P_m=2^{1-m^2}\sqrt{\pi}^{m}\prod_{i=0}^m (2i)!$ is as in \cref{ratio}, $Z_{2m}$ is the normalization constant \cref{selberg} and where $R_{2j}(u)=2^{-2j} H_{2j}(u)$ and $R_{2j+1}(u)=2^{-(2j+1)}(H_{2j+1}(u)-H_{2j}'(u))$.
Using the multilinearity of the determinant we get
$$\mean\det(Q-u\mathbbm{1})^2 =\frac{(2m)!\,P_m}{Z_{2m}} \sum_{j=0}^m \frac{2^{-2j-2}}{(2j)!}\,
\det X_j(u).$$
By \cref{ratio} we have $\tfrac{P_m}{Z_{2m}}=  2^{1-2m}$. Putting everything together yields the first claim.

In the case $k=2m+1$ we get from \cite[(22.2.39)]{mehta} that
  $$\mean\det(Q-u\mathbbm{1})^2 =\frac{(2m+1)!\,P_m}{Z_{2m+1}} \sum_{j=0}^m \frac{2^{{2j-1}}}{(2j)!}\,
  \det\begin{pmatrix}
  g_{2j} & R_{2j}(u)& R_{2j}'(u) \\
  g_{2j+1} & R_{2j+1}(u)& R_{2j+1}'(u) \\
  g_{2m+2} & R_{2m+2}(u)& R_{2m+2}'(u)
  \end{pmatrix},$$
where $P_m$, $R_{2j}(u), R_{2j+1}(u)$ are as above and
  $$g_i = \int_{u\in\HR} R_i(u) \exp(-\tfrac{u^2}{2})\,\d u.$$
By \cref{hermite_symmetry}, the Hermite polynomial $H_{2j+1}(u)$ is an odd function. Hence, we have $g_{2j+1}=0$. For even indices we use \cite[(7.373.2)]{Gradshteyn2015} to get $g_{2j}=2^{-2j}\sqrt{2\pi} \tfrac{(2j)!}{j!}$. By the multilinearity of the determinant:
\begin{equation}\label{hallo99}\mean\det(Q-u\mathbbm{1})^2 =\frac{\sqrt{2\pi}(2m+1)!\,P_m}{2^{2m+2}Z_{2m+1}} \sum_{j=0}^m \frac{2^{-2j-2}}{(2j)!}\,
\det Y_j(u).\end{equation}
From \cref{selberg} one obtains $Z_{2m+1} = 2 \sqrt{2}\, \Gamma(m+\tfrac{3}{2})\,Z_{2m}$, which together with \cref{ratio} implies
 $$ \frac{P_m}{Z_{2m+1}}= \frac{2^{-2m}}{\sqrt{2}\, \Gamma(m+\tfrac{3}{2})}.$$
Plugging this into \cref{hallo99} we conlude that
\begin{align}\mean\det(Q-u\mathbbm{1})^2 =\frac{\sqrt{\pi}(2m+1)!}{2^{4m+2}\,\Gamma(m+\tfrac{3}{2})} \sum_{j=0}^m \frac{2^{-2j-2}}{(2j)!}\,
\det Y_j(u).\end{align}
\end{proof}
Everything is now ready for the proof of Theorem \ref{thm:charpol}
\begin{proof}[Proof of Theorem \ref{thm:charpol}]
Due to the nature of \cref{thm_char_poly} we also have to make a distinction for this proof.

In the case $k=2m$ we use the formula from \cref{thm_char_poly} (1) to write
  $$\int_{u\in\HR}\,\mean \det(Q-u\mathbbm{1})^2 e^{-u^2}\, \d u = \frac{(2m)!}{2^{2m}} \sum_{j=0}^m \frac{2^{-2j-1}}{(2j)!}\, \int_{u\in\HR}
  \det X_j(u)\,\d u.$$
By \cref{hermite_derivative} we have $H_i'(u)=2iH_{i-1}(u)$. Hence, $X_j(u)$ can be written as
  $$ \begin{pmatrix}  H_{2j}(u) &  4jH_{2j-1}(u) \\ H_{2j+1}(u)-4jH_{2j-1}(u) & 2(2j+1)H_{2j}(u)-8j(2j-1)H_{2j-2}(u) \end{pmatrix}.$$
From \cref{ortho_rel} we can deduce that
  \begin{align*}\int_{u\in\HR}
  \det X_j(u)\,\d u &= 2(2j+1) 2^{2j}(2j)!\sqrt{\pi} + 16j^2 2^{2j-1} (2j-1)!\sqrt{\pi}\\
  &= 2^{2j+1}(2j)!\sqrt{\pi}(4j+1).\nonumber
\end{align*}
From this we see that
\begin{align}\label{X_j}
  \sum_{j=0}^m \frac{2^{-2j-1}}{(2j)!}\, \int_{u\in\HR}
  \det X_j(u)\,\d u &= \sqrt{\pi} \sum_{j=0}^m (4j+1 ) = \sqrt{\pi} \,(m+1)(2m+1).
\end{align}
and hence,
\begin{align*}\int_{u\in\HR}\,\mean \det(Q-u\mathbbm{1})^2 e^{-u^2}\, \d u
=  \frac{(2m)!}{2^{2m}} \, \sqrt{\pi} \,(m+1)(2m+1)=  \frac{(2m+2)!}{2^{2m+1}} \, \sqrt{\pi} .
\end{align*}
Plugging back in $m=\tfrac{k}{2}$ finishes the proof of the case $k=2m$.

In the case $k=2m+1$ we use the formula from \cref{thm_char_poly} (2) to see that
$$\int_u\mean\det(Q-u\mathbbm{1})^2 \,e^{-u^2}\,\d u =\frac{\sqrt{\pi}(2m+1)!}{2^{4m+2}\,\Gamma(m+\tfrac{3}{2})} \sum_{j=0}^m \frac{2^{-2j-2}}{(2j)!}\,\int_u
\det Y_j(u) \,e^{-u^2}\,\d u.$$
Note that the top right $2\times 2$-submatrix of $Y_j(u)$ is $X_j(u)$, so that $\det Y_j(u)$ is equal to
  \begin{equation}\label{hallo1} \frac{(2m+2)!}{(m+1)!}\, \det X_j(u) + \frac{(2j)!}{j!}\, \det \begin{pmatrix}
H_{2j+1}(u)-H_{2j}'(u) & H_{2j+1}'(u)-H_{2j}''(u)\\
    H_{2m+2}(u) &  H_{2m+2}'(u)
  \end{pmatrix}.\end{equation}
Because taking derivatives of Hermite polynomials decreases the index by one \cref{hermite_derivative} and because the integral over a product of two Hermite polynomials is only non-vanishing, if their indices agree, the integral of the determinant in \cref{hallo1} is only non-vanishing for $j=m$, in which case it is equal to
$$\int_{u\in\HR} H_{2m+1}(u)H_{2m+2}'(u)\,e^{-u^2}\,\d u = 2(2m+2) 2^{2m+1} (2m+1)!\,\sqrt{\pi},$$
by \cref{ortho_rel} and \cref{hermite_derivative}. Hence,
  \begin{align*}&\int_{u\in\HR}
  \det Y_j(u) \,e^{-u^2}\,\d u \\
  = & \begin{cases}
  \frac{(2m+2)!}{(m+1)!}\,\int_{u\in\HR}
  \det X_m(u) \,e^{-u^2}\,\d u + \frac{(2m)!}{m!}\, 2^{2m+2}(2m+2)!\sqrt{\pi}, &\text{ if } j=m,\\[0.2cm]
  \frac{(2m+2)!}{(m+1)!}\,\int_{u\in\HR}
  \det X_j(u) \,e^{-u^2}\,\d u, & \text{ else.}
\end{cases}
\end{align*}
We find that
  \begin{align*}
    &\sum_{j=0}^m \frac{2^{-2j-2}}{(2j)!}\,\int_u
    \det Y_j(u) \,e^{-u^2}\,\d u\\
    =&\frac{(2m+2)!}{m!}\,\sqrt{\pi} + \frac{(2m+2)!}{(m+1)!} \sum_{j=0}^m \frac{2^{-2j-2}}{(2j)!}\,\int_u
    \det X_j(u) \,e^{-u^2}\,\d u\\
    =& \frac{(2m+2)!}{m!} \,\sqrt{\pi}+ \frac{(2m+2)!}{(m+1)!} \,\frac{\sqrt{\pi}}{2} \,(m+1)(2m+1)\\
    =& \frac{\sqrt{\pi}}{2}\,\frac{(2m+3)!}{m!};
\end{align*}
the second-to-last line by \cref{X_j}. It follows that
\begin{align*}\int_{u\in\HR}\mean\det(Q-u\mathbbm{1})^2 \,e^{-u^2}\,\d u&= \frac{\sqrt{\pi}(2m+1)!}{2^{4m+2}\,\Gamma(m+\tfrac{3}{2})}\,\frac{\sqrt{\pi}}{2}\,\frac{(2m+3)!}{m!}\\&= \frac{\pi(2m+1)!(2m+3)!}{2^{4m+3}\,\Gamma(m+\tfrac{3}{2})\,m!}.\end{align*}
It is not difficult to verify that the last term is $2^{-2m-2} \sqrt{\pi}\,(2m+3)!$. Substituting $2m+1=k$ shows the assertion in this case.
\end{proof}

\subsection*{Acknowledgements}The authors wish to thank A. Agrachev, P. B\"urgisser, A. Maiorana for helpful suggestions and remarks on the paper and B.~Sturmfels for pointing out reference \cite{SSV} for \eqref{eq:cut}.

\bibliographystyle{plain}

\bibliography{Disc}

\end{document}